\documentclass[12pt]{article}
\usepackage[T1]{fontenc}
\usepackage{lmodern,amsmath,amsthm,amsfonts,amssymb,graphicx,float,microtype,thmtools,mathtools,xurl,multirow,tabularx,array,booktabs}
\usepackage[dvipsnames,svgnames,table]{xcolor}
\usepackage[shortlabels]{enumitem}
\setlist[itemize]{topsep=0ex,itemsep=0ex,parsep=0ex}
\setlist[enumerate]{topsep=0ex,itemsep=0ex,parsep=0ex}
\usepackage{pgf,tikz,ifthen}
\usetikzlibrary{arrows}
\usetikzlibrary{calc}
\usepackage[unicode=true]{hyperref}
\hypersetup{
colorlinks,
breaklinks=true,
linkcolor={blue!60!black},
citecolor={black},
urlcolor={blue!60!black},
pdftitle={A Two-Player Zero Forcing Game}}
\usepackage[capitalise, compress, nameinlink, noabbrev]{cleveref}
\crefname{lem}{Lemma}{Lemmas}
\crefname{thm}{Theorem}{Theorems}
\crefname{prop}{Proposition}{Propositions}
\crefname{cor}{Corollary}{Corollaries}

\usepackage[longnamesfirst,numbers,sort&compress]{natbib}
\makeatletter
\def\NAT@spacechar{~}
\makeatother
\usepackage[margin=30mm]{geometry}
\renewcommand{\baselinestretch}{1.15}
\allowdisplaybreaks

\renewcommand{\epsilon}{\varepsilon}

\renewcommand{\ge}{\geqslant}
\renewcommand{\le}{\leqslant}
\renewcommand{\geq}{\geqslant}
\renewcommand{\leq}{\leqslant}

\newcommand{\FF}{\mathcal{F}}

\renewcommand{\thefootnote}{\fnsymbol{footnote}}
\theoremstyle{plain}
\newtheorem{thm}{Theorem}
\newtheorem{lem}[thm]{Lemma}
\newtheorem{cor}[thm]{Corollary}
\newtheorem{prop}[thm]{Proposition}
\theoremstyle{definition}
\newtheorem{conj}[thm]{Conjecture}
\newtheorem{remark}{Remark}
\newtheorem{example}{Example}
\newtheorem{question}{Question}
\title{A Two-Player Zero Forcing Game}

\author{%
Dickson Y. B. Annor\,\footnotemark[1] \qquad
Ben Howerton\,\footnotemark[2]
}
\date{}

\begin{document}

\maketitle

\begin{abstract}
We introduce a competitive two-player zero forcing game on a connected graph.
Alice and Bob alternately seed white vertices or perform legal zero forces in
their own colours, and each player seeks to minimise their own number of seeds.
A force preservation rule prevents avoidable blocking of an opponent's
established force. Because distinct continuations can be equally good for the
player to move, optimal play is defined by a set-valued backward induction, and
\(Z_g(G)\) is the minimum total number of seeds among the resulting optimal
outcomes. We prove that \(Z_g(G)\geq Z(G)\), determine \(Z_g\) for paths,
cycles, stars, complete graphs, and complete bipartite graphs, and characterise
the graphs with
\(Z_g(G)=2\) by an alternating two-chain forcing schedule. We also show that
\(Z_g\) is not minor-monotone and that edge subdivision can either increase
or decrease the parameter. Exact computation verifies \(Z_g(G)\leq2Z(G)\)
through order nine.
\end{abstract}



\footnotetext[1]{Department of Mathematical and Physical Sciences, La Trobe University, Bendigo, Australia (\texttt{d.annor@latrobe.edu.au}).
}
\footnotetext[2]{\texttt{bhowerton@pointdynamics.com}.
}

\renewcommand{\thefootnote}{\arabic{footnote}}

\section{Introduction}

Let \(G\) be a graph whose vertices are initially white. If a coloured vertex
\(u\) has exactly one white neighbour \(v\), then \(u\) may \emph{force}
\(v\), written \(u\to v\), by colouring it. A set of initially coloured
vertices that can eventually colour all the remaining white vertices of \(G\) is a \emph{zero forcing set},
and the minimum size of such a set is the \emph{zero forcing number}
\(Z(G)\). Zero forcing arose in the study of minimum rank and maximum nullity
and has since developed into a flexible framework for graph propagation; see,
for example, \cite{aim2008zero,taklimi2013zero}.

Classical zero forcing is cooperative: the initial set and every subsequent
force serve the same objective. Here two players, Alice and Bob, alternately
colour vertices in their own colours. On each turn a player either pays for a
new seed or performs one of their legal forces for free. Each player seeks to
minimise their own seed count. The resulting tension is not merely between
seeding and forcing. A move may also destroy a force that the opponent has
already created. To prevent the game from degenerating into avoidable blocking,
we impose a Force Preservation Rule while allowing a player to execute any of
their own legal forces.

A related \(q\)-analogue of zero forcing is usually denoted by \(Z_q(G)\) and
involves a player and an oracle \cite{butler2015zero,fallat2023qanalogue}. Another
robustness-based variant is leaky forcing. For a fixed \(\ell\), one seeks an
initial set that can still force the graph when the locations of \(\ell\)
vertices unable to perform forces are chosen adversarially
\cite{alameda2024leaky}. Unlike the present game, leaky forcing is a one-player
process: the obstruction comes from disabled forcing vertices rather than an
opponent who colours vertices and minimises a separate seed count.
The invariant studied
here is different: it comes from a competitive two-colour game, and we denote it
by \(Z_g(G)\), following the usual convention for game graph parameters.
Because each player minimises a different coordinate, optimal continuation need
not be unique. We therefore retain every tied optimal outcome by backward
induction and define \(Z_g(G)\) as the least total seed count among those
outcomes.

An interactive implementation of the game, including a guided tutorial,
computer self-play, and exact analysis, is available at
\url{https://zf.pointdynamics.com}.

Our principal results are as follows. We prove the general lower bound
\[
Z_g(G)\geq Z(G),
\]
determine \(Z_g\) for paths, cycles, stars, complete graphs, and complete
bipartite graphs, and characterise the equality \(Z_g(G)=2\) through
alternating forcing chains.
We show that vertex deletion and edge deletion can each change \(Z_g\) in
either direction, so \(Z_g\) is not minor-monotone. We also show that edge
subdivision can either increase or decrease \(Z_g\). Finally, exact computation
verifies \(Z_g(G)\leq 2Z(G)\) for every connected graph through order nine.
These results lead to questions on the gap between \(Z_g\) and \(Z\),
joins, and other structured families.

\section{A Zero Forcing Game}
We now define a two-player game based on the zero forcing rule.

\subsection*{Rules}
Let \(G\) be a connected graph with at least one edge. Initially every vertex
is white. Alice, who uses blue, and Bob, who uses red, alternate turns, with
Alice moving first. On each turn the player chooses exactly one action:
\begin{enumerate}
\item \textbf{Seed move.} Colour one white vertex with the player's colour.
\item \textbf{Force move.} If a vertex of the player's colour has exactly one
white neighbour, colour that neighbour with the player's colour.
\end{enumerate}\par If a white vertex is the target of one of the player's legal forces, colouring that vertex is treated as a force move, not as a seed move. This convention causes no loss under optimal play, since seeding the same target produces the same colouring while increasing that player's seed count.
A move \emph{blocks} a legal force \(u\to v\) if that force is no longer legal
after the move.

\medskip
\noindent\textbf{Force Preservation Rule.}
At the beginning of a turn, consider the opponent's legal forces.
\begin{enumerate}[(i)]
\item The player may always perform one of their own legal forces, even when
that move blocks a competing force of the opponent.
\item A seed move that blocks an opponent's legal force is permitted only when
every available seed move blocks at least one of the opponent's legal forces.
\end{enumerate}
Thus an opponent's established forces must be preserved whenever a nonblocking
seed is available, but they never prevent the player from using a force of
their own.

The game ends when every vertex has been coloured. Let \(S_A(G)\) and
\(S_B(G)\) be the seed sets used by Alice and Bob. Each player seeks to minimise
the size of their own seed set. Alice wins when
\(\lvert S_A(G)\rvert<\lvert S_B(G)\rvert\), Bob wins when the reverse
inequality holds, and the game is a draw when the two counts are equal.

\begin{example}\label{exam:K13}
Let \(c\) be the centre of \(K_{1,3}\), with leaves \(v_1,v_2,v_3\).
Suppose Alice seeds \(v_1\), creating the force \(v_1\to c\). Bob may not seed
\(c\), because seeding \(v_2\) or \(v_3\) preserves Alice's force. If Bob seeds
\(v_2\), however, then both players can force \(c\). Alice may execute
\(v_1\to c\) on her next turn even though this removes Bob's competing force
\(v_2\to c\). The first prohibition illustrates part (ii) of the rule, while
the competing force illustrates part (i).
\end{example}

\subsection{Optimal play and the game zero forcing number}
Because each player minimises their own seed count, rather than a common
payoff, an optimal move need not be unique. We therefore use the following
set-valued backward-induction convention.

For every game position \(P\), let \(\mathcal O(P)\) be a set of final
seed-count pairs. If \(P\) is terminal with seed counts \((s_A,s_B)\), set
\(\mathcal O(P)=\{(s_A,s_B)\}\). Otherwise, for each legal move \(m\), let
\(P_m\) be the resulting position and put
\[
\mathcal U(P)=\bigcup_{m\text{ legal}}\mathcal O(P_m).
\]
If Alice is to move, retain from \(\mathcal U(P)\) all pairs whose first
coordinate is minimum. If Bob is to move, retain all pairs whose second
coordinate is minimum. Thus tied continuations remain optimal:
\[
\mathcal O(P)=
\begin{cases}
\{(s_A,s_B)\in\mathcal U(P):
  s_A=\min\{x:(x,y)\in\mathcal U(P)\}\},&\text{Alice to move},\\[1mm]
\{(s_A,s_B)\in\mathcal U(P):
  s_B=\min\{y:(x,y)\in\mathcal U(P)\}\},&\text{Bob to move}.
\end{cases}
\]
Let \(P_0(G)\) be the initial position and write
\(\mathcal O(G)=\mathcal O(P_0(G))\). The \emph{game zero forcing number of
\(G\)} is
\[
Z_g(G)=\min\{s_A+s_B:(s_A,s_B)\in\mathcal O(G)\}.
\]
This set-valued recursion is an explicit solution convention for the game. It
selects terminal pairs directly from the union of optimal child outcomes and
does not impose any additional preference on a player's tied outcomes. Since
the game tree is finite, \(\mathcal O(G)\) is nonempty and the minimum exists.

This parameter may be viewed as a game-theoretic analogue of the classical zero forcing number $Z(G)$.

\begin{remark}
Unlike many graph games whose outcomes are necessarily win or lose
(see \cite{brevsar2010domination,brevsar2021independence,gardner1981mathematical}),
this game naturally admits three descriptive outcomes: Alice wins, Bob wins,
or the game ends in a draw. This classification does not affect the solution
concept. Optimality is determined solely by minimising the number of seeds
used by the player whose turn it is. Thus a player can be indifferent between
a win, a draw, and a loss when their own seed count is the same.
\end{remark}

\begin{figure}
    \centering
    \begin{tikzpicture}[scale=1]
\begin{scope}[shift ={(9.5, 0)}]
  \draw[ thick] (0,0) -- (3,2.5);   
\draw[ thick] (3,2.5) -- (6,0);
\draw[ thick] (1,-3) -- (5,-3);
\draw[ thick] (0,0) -- (1,-3);
\draw[ thick] (5,-3) -- (6,0);
\draw[ thick] (0,0) -- (1.5,0);
\draw[ thick] (3,2.5) -- (3,1);
\draw[ thick] (6,0) -- (4.5,-0.5);
\draw[ thick] (5,-3) -- (4,-2);
\draw[ thick] (1,-3) -- (2,-2);
\draw[ thick] (1.5,0) -- (4.5,-0.5);
\draw[ thick] (4.5,-0.5) -- (2,-2);
\draw[ thick] (2,-2) -- (3,1);
\draw[ thick] (3,1) -- (4,-2);
\draw[ thick] (4,-2) -- (1.5,0);
 \filldraw[thick, fill=red] (0,0) circle (4pt) node[anchor=east]{$u_1$};
  \filldraw[thick, fill=white] (3,2.5) circle (4pt) node[anchor=south]{$u_2$};
   \filldraw[thick, fill=white] (6,0) circle (4pt) node[anchor=west]{$u_3$};
\filldraw[thick, fill=blue] (5,-3) circle (4pt) node[anchor=north]{$u_4$};
 \filldraw[thick, fill=blue] (1,-3) circle (4pt) node[anchor=north]{$u_5$};
  \filldraw[thick, fill=red] (1.5,0) circle (4pt) node[anchor=south]{$v_1$};
   \filldraw[thick, fill=blue] (4.5,-0.5) circle (4pt) node[anchor=south]{$v_3$};
 \filldraw[thick, fill=white] (2,-2) circle (4pt) node[anchor=north]{$v_5$};
  \filldraw[thick, fill=white] (3,1) circle (4pt) node[anchor=west]{$v_2$};
   \filldraw[thick, fill=white] (4,-2) circle (4pt) node[anchor=west]{$v_4$};  
\end{scope}

    \end{tikzpicture}
    \caption{The Petersen graph.}
    \label{fig:petersencover}
\end{figure}
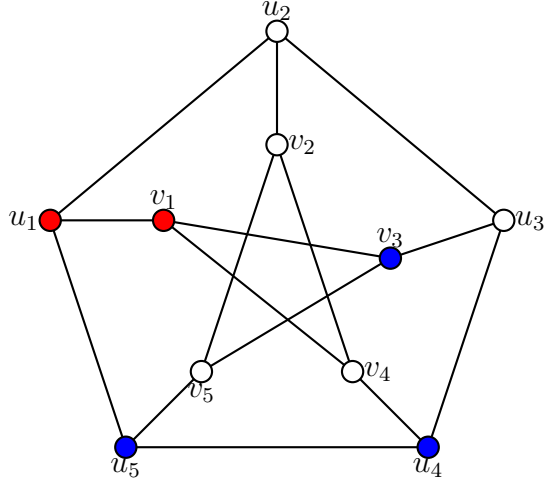

\section{Examples and basic results}
\begin{example}
In Figure~\ref{fig:petersencover}, the blue vertices and the red vertices are
the seeds chosen by Alice and Bob, respectively. Alice seeds \(u_5\), Bob
seeds \(u_1\), Alice seeds \(u_4\), and Bob seeds \(v_1\). Alice then
performs \(u_5\to v_5\), colouring \(v_5\) blue, and Bob performs
\(u_1\to u_2\), colouring \(u_2\) red. Alice must use her third seed and
chooses \(v_3\), colouring it blue. The remaining moves are
\[
v_1\to v_4\quad\text{(Bob)},\qquad
v_3\to u_3\quad\text{(Alice)},\qquad
u_2\to v_2\quad\text{(Bob)}.
\]
Thus Alice uses three seeds and Bob uses two. In fact, exact backward
induction gives \(Z_g(\text{Petersen graph})=5\).
\end{example}

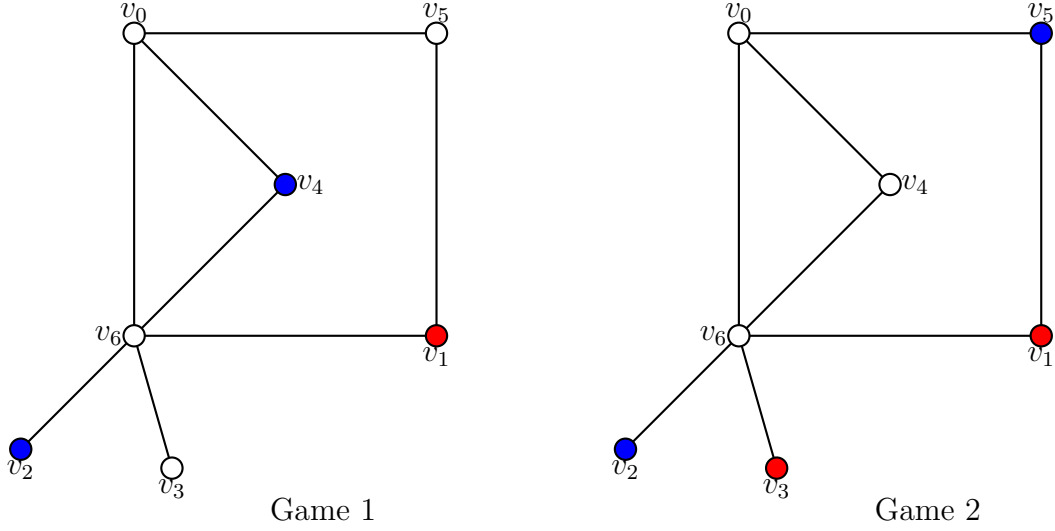
\begin{figure}
    \centering

\begin{tikzpicture}
\begin{scope}[shift ={(-5, 0)}]
\draw[ thick] (0,0) -- (4,0); 
\draw[ thick] (0,0) -- (0,4);
\draw[ thick] (4,0) -- (4,4);
\draw[ thick] (0,4) -- (4,4);
\draw[ thick] (0,4) -- (2,2);
\draw[ thick] (0,0) -- (2,2);
\draw[ thick] (0,0) -- (-1.5,-1.5);
\draw[ thick] (0,0) -- (0.5,-1.75);
  \filldraw[thick, fill=white] (0,0) circle (4pt) node[anchor=east]{$v_6$};
   \filldraw[thick, fill=white] (0,4) circle (4pt) node[anchor=south]{$v_0$};
 \filldraw[thick, fill=white] (4,4) circle (4pt) node[anchor=south]{$v_5$};
  \filldraw[thick, fill=red] (4,0) circle (4pt) node[anchor=north]{$v_1$};
   \filldraw[thick, fill=blue] (2,2) circle (4pt) node[anchor=west]{$v_4$};
    \filldraw[thick, fill=blue] (-1.5,-1.5) circle (4pt) node[anchor=north]{$v_2$};
     \filldraw[thick, fill=white] (0.5,-1.75) circle (4pt) node[anchor=north]{$v_3$};
      \filldraw[thick, fill=white] (2.5,-2) circle (0pt) node[anchor=north]{Game 1};
\end{scope}

\begin{scope}[shift ={(3, 0)}]
\draw[ thick] (0,0) -- (4,0); 
\draw[ thick] (0,0) -- (0,4);
\draw[ thick] (4,0) -- (4,4);
\draw[ thick] (0,4) -- (4,4);
\draw[ thick] (0,4) -- (2,2);
\draw[ thick] (0,0) -- (2,2);
\draw[ thick] (0,0) -- (-1.5,-1.5);
\draw[ thick] (0,0) -- (0.5,-1.75);
  \filldraw[thick, fill=white] (0,0) circle (4pt) node[anchor=east]{$v_6$};
   \filldraw[thick, fill=white] (0,4) circle (4pt) node[anchor=south]{$v_0$};
 \filldraw[thick, fill=blue] (4,4) circle (4pt) node[anchor=south]{$v_5$};
  \filldraw[thick, fill=red] (4,0) circle (4pt) node[anchor=north]{$v_1$};
   \filldraw[thick, fill=white] (2,2) circle (4pt) node[anchor=west]{$v_4$};
    \filldraw[thick, fill=blue] (-1.5,-1.5) circle (4pt) node[anchor=north]{$v_2$};
     \filldraw[thick, fill=red] (0.5,-1.75) circle (4pt) node[anchor=north]{$v_3$};
     \filldraw[thick, fill=white] (2.5,-2) circle (0pt) node[anchor=north]{Game 2};
\end{scope}

\end{tikzpicture}
\caption{Two legal plays on the same graph $H$.}
    \label{fig:2optimalgames}
\end{figure}

\begin{example}
Figure~\ref{fig:2optimalgames} illustrates that different legal plays on the
same graph can use different totals. In Game~1, Alice seeds \(v_2\), Bob seeds
\(v_1\), Alice performs \(v_2\to v_6\), and Bob performs
\(v_1\to v_5\). Alice then seeds \(v_4\), Bob performs
\(v_5\to v_0\), and Alice finishes with \(v_6\to v_3\). Thus the seed-count
pair is \((2,1)\).

In Game~2, Alice seeds \(v_5\), Bob seeds \(v_1\), Alice performs
\(v_5\to v_0\), and Bob performs \(v_1\to v_6\). Alice then performs
\(v_0\to v_4\), Bob seeds \(v_3\), and Alice seeds \(v_2\). Thus the
seed-count pair is \((2,2)\). These two plays motivate the need to distinguish
legal play from optimal play.
\end{example}
\begin{example}\label{exam:minimum-necessary}
Let \(J=K_5-e\), where \(x\) and \(y\) are the nonadjacent vertices and
\(p,q,r\) are the remaining vertices. Set-valued backward induction gives
\[
\mathcal O(J)=\{(2,1),(2,2)\}.
\]
The pair \((2,1)\) is realised by
\[
\begin{array}{c|ccccc}
\text{turn}&1&2&3&4&5\\ \hline
\text{move}&
\text{Alice seeds }p&
\text{Bob seeds }x&
\text{Alice seeds }q&
x\to r&
p\to y,
\end{array}
\]
while \((2,2)\) is realised when Alice seeds \(x\), Bob seeds \(y\), Alice
seeds \(p\), Bob seeds \(q\), and Alice forces \(x\to r\). A direct check of
the five-vertex game tree shows that both continuations survive the recursive
optimality test above. Their totals are \(3\) and \(4\), respectively, so
\[
Z_g(J)=3.
\]
This proves that the minimum in the definition is necessary.
\end{example}

\begin{prop}\label{prop:force-count}
Let \(G\) be a connected graph of order \(n\ge2\). Every completed game on
\(G\) consists of exactly \(n\) moves. If \(F_A(G)\) and \(F_B(G)\) denote the
numbers of force moves performed by Alice and Bob, respectively, then
\[
|S_A(G)|=\left\lceil\frac n2\right\rceil-F_A(G),
\qquad
|S_B(G)|=\left\lfloor\frac n2\right\rfloor-F_B(G).
\]
Consequently, if \(F(G)=F_A(G)+F_B(G)\), then
\[
F(G)=n-\bigl(|S_A(G)|+|S_B(G)|\bigr).
\]
\end{prop}

\begin{proof}
Each move colours exactly one previously white vertex, and no coloured vertex
ever changes colour. Thus every completed game has \(n\) moves. Alice makes
the odd-numbered moves and Bob the even-numbered moves, so they have
\(\lceil n/2\rceil\) and \(\lfloor n/2\rfloor\) turns, respectively. Each of
a player's turns is either a seed or a force, which gives the two displayed
identities and their sum.
\end{proof}

\begin{cor}\label{cor:force-count}
For each player, minimising their own number of seeds is equivalent to
maximising their own number of forces. Among games that arise under optimal
play, a game realises \(Z_g(G)\) if and only if it maximises the total number
of force moves.
\end{cor}

\begin{lem}\label{lem:cup}
Let $G$ be a connected graph, and consider any completed game on $G$.
Then the set
\[
S_A(G)\cup S_B(G)
\]
is a zero forcing set of $G$.
\end{lem}

\begin{proof}
Start an ordinary zero forcing process with every vertex of
\(S_A(G)\cup S_B(G)\) black, including seeds that were played later in the
game. Replay the game's force moves in their original order and ignore the two
colours. The target of a game force is never a seed, so it is still white when
that force is reached in the replay. Precolouring later seeds can only remove
other white neighbours of the forcing vertex. Hence every replayed force is a
valid ordinary zero force.

The replay colours every vertex that was not a seed, so
\(S_A(G)\cup S_B(G)\) is a zero forcing set of \(G\).
\end{proof}

\begin{lem}\label{lem:2}
Let $G$ be a connected graph of order $n\ge 2$. In every play of the game,
\[
|S_A(G)|\geq 1 \quad\text{and}\quad |S_B(G)|\geq 1.
\]
Consequently,
$
Z_g(G)\geq 2$.

\end{lem}

\begin{proof}
Since all vertices of $G$ are initially white, Alice cannot make a force move on her first turn. Therefore, her first move must be a seed move, and hence
$
|S_A(G)|\geq 1$.

Now consider Bob's first turn. Because $G$ has at least one edge, it has at least two vertices. Therefore, the game cannot end after Alice's first move, since a seed move colors only one vertex. Hence, Bob receives at least one turn.

At the beginning of Bob's first turn, no vertex is coloured red, because Bob has not yet made a move. A force move by Bob requires a red vertex with exactly one white neighbour. Thus, Bob has no legal force move available on his first turn and must make a seed move. Therefore, $
|S_B(G)|\geq 1$.

It follows that every play of the game satisfies
$
|S_A(G)|+|S_B(G)|\geq 2$. 
Since $Z_g(G)$ is defined as the minimum total number of seeds used among all optimal-play outcomes, we obtain
$ Z_g(G)\geq 2$.
\end{proof}

\begin{lem}\label{lem:one-one-retention}
Let \(G\) be a connected graph of order $n\ge 2$. If a legal play has
outcome \((1,1)\), then \((1,1)\in\mathcal O(G)\) and \(Z_g(G)=2\).
\end{lem}

\begin{proof}
Work backwards along the given play. Its terminal position has outcome
\((1,1)\). Suppose this pair has been retained at the position following the
next move on the play. It then belongs to the union of optimal child outcomes
at the preceding position. By Lemma~\ref{lem:2}, every terminal outcome has
both coordinates at least \(1\). The coordinate belonging to the player who
moves at the preceding position is therefore already minimum, so the
backward-induction rule retains \((1,1)\) there as well. Induction back to the
initial position gives \((1,1)\in\mathcal O(G)\), and Lemma~\ref{lem:2}
then gives \(Z_g(G)=2\).
\end{proof}

The following bound is immediate. 

\begin{thm}\label{prop:ZgZ}
For every connected graph $G$ with order $n\ge 2$,
$Z_g(G)\ge Z(G)$. 
\end{thm}

\begin{proof}
Consider a completed game on $G$ arising from an optimal play and let $S=S_A(G) \cup S_B(G)$. From Lemma~\ref{lem:cup}, $S$ is a zero forcing set of $G$. Therefore
$$
|S_A(G)|+|S_B(G)| = |S| \ge Z(G).
$$
Since this inequality holds for every game arising from optimal play, taking
the minimum over all such games gives
$Z_g(G)\ge Z(G)$.
\end{proof}

The bound in Theorem~\ref{prop:ZgZ} is best possible. For example, the Petersen graph in Figure~\ref{fig:petersencover} satisfies $Z_g(\text{Petersen graph}) = Z(\text{Petersen graph}) = 5$.

\begin{cor}
For every connected graph $G$ with order $n\ge 2$,
\[
Z_g(G)\ge \max\{\operatorname{tw}(G),\,P(G)\},
\]
where \(\operatorname{tw}(G)\) and \(P(G)\) denote the treewidth and
path cover number of \(G\), respectively. These are inherited from the
classical inequalities
\(\operatorname{tw}(G)\le Z(G)\) and \(P(G)\le Z(G)\); see, for example,
\cite{taklimi2013zero}.
\end{cor}

Next, we compute $Z_g(G)$ for some families of graphs.

\begin{prop}\label{prop:path}
For every path \(P_n\) with \(n\geq 2\), 
$
Z_g(P_n)=2.
$
\end{prop}
\begin{proof}
Label the vertices \(v_1,\ldots,v_n\) in order. Alice can seed \(v_1\), after
which Bob can seed \(v_n\). Both players can then use forces only, advancing
from the two ends until the path is coloured. This legal play has outcome
\((1,1)\), so Lemma~\ref{lem:one-one-retention} gives \(Z_g(P_n)=2\).
\end{proof}

\begin{prop}\label{prop:cyc}
If $C_n$ is a cycle of length $n\geq 3$, then
$ Z_g(C_n)=2$.
\end{prop} 
\begin{proof}
Let \(v\) be Alice's first seed, and let Bob seed a neighbour \(w\) of
\(v\). Alice can force from \(v\) away from \(w\). If any vertices remain,
Bob can force from \(w\) away from \(v\), and the two forcing chains then
advance in opposite directions until the cycle is coloured. (For \(C_3\),
Alice's first force finishes the game.) This legal play has outcome \((1,1)\),
so Lemma~\ref{lem:one-one-retention} gives \(Z_g(C_n)=2\).
\end{proof}

\begin{prop}\label{prop:star}
Let \(n\geq3\). Then
\[
Z_g(K_{1,n})=
\begin{cases}
n, & \text{if }n\text{ is odd},\\[2mm]
n-1, & \text{if }n\text{ is even}.
\end{cases}
\]
\end{prop}

\begin{proof}
Let \(c\) be the centre and let
\(L=\{\ell_1,\ldots,\ell_n\}\) be the leaves, and write \(F_A\) for
Alice's number of force moves. Every force in a star is either a leaf forcing
\(c\), which can occur at most once, or \(c\) forcing
the final white leaf, which can also occur at most once. Hence Alice performs
at most two forces in any play. If \(n=2k+1\), then the star has even order
and Bob makes the final move. In that case Alice cannot perform both types of
force: after Alice colours \(c\), its force becomes available only when one
white leaf remains, immediately before Bob's final move. Bob may then seed
that last leaf because no nonblocking seed target exists. Thus
\[
F_A\leq
\begin{cases}
2,&n=2k,\\
1,&n=2k+1.
\end{cases}
\]
Proposition~\ref{prop:force-count} therefore gives
\[
|S_A|\geq
\begin{cases}
k-1,&n=2k,\\
k,&n=2k+1.
\end{cases}
\]

Alice attains these bounds by seeding a leaf \(\ell_1\). This creates the
force \(\ell_1\to c\), so the Force Preservation Rule prevents Bob from
seeding \(c\). Bob has no force and must seed another leaf, say \(\ell_2\),
after which Alice performs \(\ell_1\to c\). At every Alice position on this
continuation, following the displayed strategy attains her global lower bound
above, so the continuation is retained. At every Bob position, Bob has no
force and all legal seed choices are symmetric.

There are now \(n-2\) white leaves and it is Bob's turn. If \(n=2k\),
the players make \(2k-3\) further seed moves before Alice forces the final
leaf. The retained outcome is \((k-1,k)\), with total \(n-1\). If
\(n=2k+1\), all \(2k-1\) remaining leaves are seeded, giving the retained
outcome \((k,k+1)\), with total \(n\). The stated formula follows.
\end{proof}

\begin{prop}\label{prop:complete}
For $n\geq 3$,
$
Z_g(K_n)=Z(K_n)=n-1$.
\end{prop}

\begin{proof}
In \(K_n\), every coloured vertex is adjacent to every white vertex. Hence no
force is possible while at least two white vertices remain. The first \(n-1\)
moves are therefore necessarily seeds, independent of the players' choices.
When one white vertex remains, any coloured vertex can force it, so every
completed game uses exactly \(n-1\) seeds. Thus
\[
Z_g(K_n)=n-1=Z(K_n).
\]
\end{proof}

\subsection{Complete bipartite graphs}\label{subsec:complete-bipartite}

This subsection determines the game zero forcing number of every complete bipartite graph.
The proof is most transparent in terms of force counts. By
Corollary~\ref{cor:force-count}, the backward-induction rule may equivalently
be expressed by retaining the outcomes that maximise the mover's own number
of forces. For a position \(P\), let \(\FF(P)\) denote the resulting retained
set of future force-count pairs. The number of future turns belonging to each
player is fixed by the number of white vertices and the player to move, so
this force-count formulation is equivalent at every nonterminal position, not
only at the root.

The following elementary observation will be used repeatedly.

\begin{lem}[Guarantees survive backward induction]\label{lem:guarantee}
Suppose a player $P$ has a strategy from a position $R$ that guarantees at
least $r$ of $P$'s future force moves against every legal response.  Then every
pair in $\FF(R)$ gives $P$ at least $r$ future forces.  In particular, if $P$
can guarantee the largest number of forces possible in the whole game, every
outcome retained at $R$ gives $P$ that number.
\end{lem}

\begin{proof}
Fix a guaranteeing strategy $\sigma$ for $P$.  Call a node
\emph{$\sigma$-consistent} if it can be reached when $P$ follows $\sigma$ and
the opponent makes arbitrary legal moves.  For each such node $Q$, let $g(Q)$
be the number of additional forces that $\sigma$ guarantees to $P$ from $Q$.
We prove by backward induction that every pair retained at $Q$ has
$P$-coordinate at least $g(Q)$.

The assertion is immediate at a terminal node.  Suppose first that $P$ moves
at $Q$.  Let $Q'$ be the child chosen by $\sigma$, and include in the force
coordinate the possible force made on the edge from $Q$ to $Q'$.  By induction,
every pair retained at $Q'$ gives $P$ at least the residual guarantee there, so
the union considered at $Q$ contains a pair giving $P$ at least $g(Q)$.  The
backward-induction rule retains exactly the pairs with maximal $P$-coordinate.
All retained pairs at $Q$ therefore have the same $P$-coordinate, and that
coordinate is at least $g(Q)$.

Now suppose the opponent moves at $Q$.  Strict alternation implies that each
child is either a $P$-node or a terminal node, and every legal child is
$\sigma$-consistent.  The strategy guarantees $g(Q)$ against every such
choice.  By induction, every pair in the union of the children's retained
sets gives $P$ at least $g(Q)$ forces.  Filtering that union by the opponent's
coordinate can only remove pairs; it cannot invalidate the inequality for a
pair that remains.  This completes the induction.  Applying it at $R$ gives
the result.
\end{proof}

\subsubsection{The two-force cap}

Let the parts of $K_{m,n}$ be $X$ and $Y$.  A force whose target is in $X$ can
occur only when $X$ has exactly one white vertex.  Once that force is made,
$X$ is full and no later force can have a target in $X$.  The same holds for
$Y$.  Hence every play satisfies
\begin{equation}\label{eq:two-force-cap}
F_A+F_B\leq 2.
\end{equation}

The ordinary zero forcing number is standard; see, for example, \cite{aim2008zero,taklimi2013zero}:
\begin{equation}\label{eq:classical}
Z(K_{m,n})=m+n-2.
\end{equation}
Indeed, colouring all but one vertex in each part is a zero forcing set.
Conversely, with at least three vertices initially white, either both parts
contain at least two white vertices, so no force is possible, or one part has
one white vertex and the other has at least two.  In the latter case the
singleton part may be filled, but the process then stops with at least two
white vertices in the other part.

It remains to prove that optimal play retains an outcome with two force moves.

\subsubsection{A two-pile parity strategy}

The next lemma is the central strategy argument.  The two bipartition classes
are viewed as piles of white vertices.

\begin{lem}[Two-pile parity lemma]\label{lem:parity}
Consider a position on $K_{m,n}$ satisfying the following conditions.
\begin{enumerate}
\item Both parts contain at least two white vertices.
\item A designated player $D$ owns at least one coloured vertex in each part.
\item Player $D$ is scheduled to make the final move of the game.
\end{enumerate}
Then $D$ can guarantee both remaining force moves.
\end{lem}

\begin{proof}
First suppose the opponent moves next. Since both parts have at least two
white vertices, no force exists and the opponent must seed. If that move
leaves exactly one white vertex in a part, $D$ forces that vertex immediately.
Otherwise it is now $D$'s turn with at least two white vertices in each part.
Thus it is enough to describe the strategy beginning on a turn of $D$.

When $D$ moves with at least two white vertices in each part, the number of
white vertices is odd because $D$ is the eventual last mover. Therefore one
part has at least three white vertices. Player $D$ seeds a vertex in such a
part. The chosen part still has at least two white vertices, and the other part
is unchanged, so both parts still have at least two white vertices.

Thereafter maintain the following invariant at the beginning of every
opponent turn: each part has at least two white vertices. If the opponent's
move leaves exactly one white vertex in a part, $D$ forces that vertex on the
next turn. This is an own legal force and is therefore permitted even if the
opponent has a competing force. If the opponent leaves at least two white
vertices in each part, it is again $D$'s turn with an odd number of white
vertices. One part has at least three white vertices, so $D$ seeds that part
and restores the invariant.

The game is finite, so eventually the opponent leaves exactly one white vertex
in a part. Player $D$ forces that vertex and thereby completes the part. All
remaining white vertices lie in the other part. Player $D$ is scheduled to
make the final move and now owns a vertex in the completed part, including the
vertex just forced. Hence the final move is also a force by $D$.
\end{proof}

Assume from now on that $2\leq m\leq n$ and put $N=m+n$.

\begin{prop}[Odd order, smaller part at least three]
\label{prop:odd-large}
If $N$ is odd and $m\geq 3$, Alice can guarantee both forces.  Thus the retained
force pair is $(2,0)$.
\end{prop}

\begin{proof}
Alice is the player scheduled to move last.  She first seeds a vertex of the
smaller part $X$.  Her second move seeds a vertex of $Y$.

If Bob's first seed was in $Y$, the remaining white counts are $m-1$ and
$n-2$.  These are at least two: $m\geq3$, and $n\geq4$ because $m+n$ is odd.
Alice owns a vertex in each part, so Lemma~\ref{lem:parity} applies.

If Bob's first seed was in $X$ and $m\geq4$, the remaining white counts are
$m-2$ and $n-1$, again at least two, and the same lemma applies.

It remains to consider $m=3$ when Bob also begins in $X$.  After Alice's second
seed, one white vertex remains in $X$.  Alice owns a vertex in $Y$ and hence
has a force to that singleton.  Bob owns no vertex in $Y$, so this is an
opponent-only force from Bob's point of view.  Since white vertices in $Y$ are
safe seed targets, the force-preservation rule compels Bob to seed $Y$.
Alice then forces the last vertex of $X$.  There are $n-2$ white vertices in
$Y$.  Since $3+n$ is odd, $n-2$ is even.  Bob moves next, so Alice makes the
last move in $Y$, and that move is a force from her coloured vertex in $X$.

Alice has therefore guaranteed two forces, the global maximum in
\eqref{eq:two-force-cap}.  Lemma~\ref{lem:guarantee} gives the retained pair
$(2,0)$.
\end{proof}

\begin{prop}[Even order, smaller part at least four]
\label{prop:even-large}
If $N$ is even and $m\geq4$, Bob can guarantee both forces after every first
move of Alice.  Thus the retained force pair is $(0,2)$.
\end{prop}

\begin{proof}
Bob is scheduled to make the final move.  After Alice's first seed, call its
part $X$ and call the other part $Y$.  Bob seeds $Y$.  On Bob's second turn he
seeds $X$, so he then owns a vertex in both parts.

If Alice's second seed was in $Y$, each part has lost two vertices.  Both still
have at least two white vertices, and Lemma~\ref{lem:parity} applies.  If
Alice's second seed was in $X$ and $|X|\geq5$, the remaining white counts in
$X$ and $Y$ are at least two, so the lemma again applies.

The only remaining case is $|X|=4$ and Alice's first two seeds both lie in
$X$.  After Bob's second seed, one white vertex remains in $X$.  Bob owns a
vertex in $Y$ and Alice does not, so Bob's force to the singleton of $X$ is
protected.  Alice must seed $Y$, after which Bob forces the last vertex of
$X$.  Since $N$ and $|X|$ are even, the number $|Y|-2$ of remaining white
vertices is even.  Alice moves next, so Bob makes and forces the last move in
$Y$.

Thus Bob guarantees the maximum of two forces after every first move of Alice.
Backward induction retains $(0,2)$.
\end{proof}

\begin{prop}[The $K_{3,n}$ even-order case]
\label{prop:three}
If $m=3$ and $N$ is even, each player can guarantee at least one force.  Hence
the retained force pair is $(1,1)$.
\end{prop}

\begin{proof}
Here $n$ is odd.  Alice seeds the three-vertex part $X$.

If Bob also seeds $X$, Alice seeds $Y$.  One white vertex remains in $X$, and
Alice's force to it is protected because Bob owns no vertex in $Y$.  Bob must
seed $Y$, after which Alice forces the last vertex of $X$.

If Bob instead seeds $Y$, Alice also seeds $Y$.  Alice then owns a vertex in
both parts, while Bob owns no vertex in $X$.  On later turns Alice plays in $Y$
whenever she has no force.  If Bob ever seeds $X$, Alice forces the remaining
vertex of $X$.  If Bob never seeds $X$, eventually $Y$ becomes a singleton.
If this occurs on Alice's turn, she forces it from her vertex in $X$.  If it
occurs on Bob's turn while he still owns no vertex in $X$, the singleton of
$Y$ is an opponent-only force target and the two white vertices of $X$ are
safe.  Bob is compelled to seed $X$, after which Alice forces the remaining
vertex of $X$.  Alice therefore guarantees one force.

Bob guarantees one force by seeding $X$ on his first turn, regardless of
Alice's first move.  Suppose first that Alice also began in $X$.  If Alice
seals $X$ by seeding its last white vertex on move three, all remaining moves
lie in $Y$.  Bob is the final player and owns a vertex in $X$, so the final
move in $Y$ is his force.  If Alice instead enters $Y$ on move three, her force
to the singleton of $X$ is an opponent-only force from Bob's perspective, and
the white vertices of $Y$ are the only safe seed targets.  The preservation
rule therefore compels Bob to seed $Y$ on move four.  In particular, Bob now
owns a vertex in each part.  If Alice takes the force in $X$, Bob makes the
final force in $Y$; if Alice delays it, Bob may take the force in $X$ on his
next turn.  Thus both ways in which Alice can complete $X$, by a seed on move
three or by a later force, leave Bob a force.

Suppose instead that Alice begins in $Y$, and Bob again seeds $X$.  If Alice
next seeds $X$, then either she later seals $X$, leaving Bob the final force in
$Y$, or the singleton of $X$ survives until Bob can force it.  If Alice next
seeds $Y$, then for $n\geq5$ Bob seeds $Y$ as well.  At that point at least two
white vertices remain in each part, Bob owns both parts, and
Lemma~\ref{lem:parity} applies.  When $n=3$, Bob instead already has a force
from his vertex in $X$ to the final white vertex of $Y$ and takes it.  Hence
Bob obtains at least one force on every legal line.  Notice that whenever the
singleton of $X$ survives beyond move three, the preservation restriction has
put a Bob-owned vertex in $Y$ by move four; this is what makes Bob's later
force to $X$ available.

At each of Bob's first-turn positions, backward induction retains a force
count of at least one for Bob.  Alice's initial strategy guarantees at least
one for Alice.  Together with the two-force cap, this forces the pair $(1,1)$.
\end{proof}

\subsubsection{The exceptional family \texorpdfstring{$K_{2,n}$}{K2,n}}

Let $|X|=2$ and $|Y|=n$.  The first two moves are necessarily seeds.  For a
two-letter word such as $XY$, the first letter denotes Alice's first part and
the second denotes Bob's first part.  No force has yet occurred, so the table
below records the retained force pair from the resulting position.

\begin{prop}\label{prop:two-table}
For $n\geq4$, set-valued backward induction gives
\begin{center}
\begin{tabular}{c cc}
\toprule
First two seed parts & $n$ odd & $n$ even\\
\midrule
$XX$ & $(1,0)$ & $(0,1)$\\
$XY$ & $(1,1)$ & $(0,2)$\\
$YX$ & $(2,0)$ & $(1,1)$\\
$YY$ & $(1,1)$ & $(1,1)$\\
\bottomrule
\end{tabular}
\end{center}
\end{prop}

\begin{proof}
After $XX$, the part $X$ is full and all $n$ white vertices lie in $Y$.
Alice moves next.  Both players own a vertex in $X$, so the player who makes
the last move in $Y$ obtains its single force.  This is Alice for odd $n$ and
Bob for even $n$.

After $XY$, one white vertex remains in $X$, and Bob owns a vertex in $Y$.
Thus Bob has a force to the singleton of $X$.  Alice owns no vertex in $Y$, and
the at least three white vertices of $Y$ are safe seed targets.  The
force-preservation rule compels Alice to seed $Y$.  Bob can then force the last
vertex of $X$.  If $n$ is even, Bob also makes the last move among the remaining
$n-2$ vertices of $Y$, so he obtains both forces.  If $n$ is odd, Alice makes
the final force in $Y$, so the pair is $(1,1)$.  In the odd case Bob cannot
obtain two: if he delays the force in $X$, Alice can take it on her next turn
and still has the final move in $Y$.  Hence Bob's immediate force is optimal.

After $YX$, Alice has a force to the singleton of $X$.  Taking it immediately
gives her both forces when $n$ is odd and gives one force to each player when
$n$ is even.  Alice may instead delay once by seeding $Y$.  At this first
delay, her force to $X$ is an opponent-only force for Bob, so the preservation
rule compels Bob to seed $Y$ in reply.  Alice now has a second decision point,
but the position has changed: Bob owns a vertex in $Y$ and therefore has his
own force to the singleton of $X$.

If Alice forces $X$ at this second decision point, she again obtains $(2,0)$
when $n$ is odd and $(1,1)$ when $n$ is even, by the parity of the remaining
vertices in $Y$.  If Alice delays a second time, Bob forces $X$ on his next
turn. Both vertices of $X$ are then Bob's. Alice owns no vertex in $X$, so her
vertices in $Y$ can never force, and she finishes with zero forces. Thus Alice
strictly prefers to force at her first or second decision point.  In either case the displayed pair is retained: $(2,0)$ for
odd $n$ and $(1,1)$ for even $n$.

It remains to analyse $YY$.  Let $Q_n$ denote this position.  Both players own
a vertex in $Y$, both vertices of $X$ are white, $n-2$ vertices of $Y$ are
white, and Alice moves.  The claim is
\begin{equation}\label{eq:Qn}
\FF(Q_n)=\{(1,1)\}\qquad(n\geq4).
\end{equation}
The proof is by induction in steps of two.

For $Q_4$, if Alice seeds $X$, Bob forces the remaining vertex of $X$ and also
makes the final force in $Y$, giving $(0,2)$.  If Alice seeds $Y$, Bob has two
relevant continuations.  He may seed $X$, after which Alice forces the other
vertex of $X$ and Bob forces the last vertex of $Y$, giving $(1,1)$.  Or he may
seed the last vertex of $Y$, after which play in $X$ gives $(0,1)$.  Bob has one
force in both continuations, so both survive at his node.  Alice retains the
continuation with her own force, giving $\FF(Q_4)=\{(1,1)\}$.

For $Q_5$, Alice may seed $X$.  Bob optimally forces the remaining vertex of
$X$, and Alice makes the final force among the three remaining vertices of
$Y$, giving $(1,1)$.  If Bob delays by seeding $Y$, Alice forces $X$ and Bob
gets no force, so Bob strictly prefers the immediate force.

If Alice instead begins by seeding $Y$, Bob has two part choices.  If Bob seeds
$X$, Alice immediately forces the remaining vertex of $X$.  Two white vertices
then remain in $Y$, Bob moves next, and Alice makes the final force in $Y$.
This gives $(2,0)$, so it cannot be Bob's optimal reply.

If Bob also seeds $Y$, the remaining position has two white vertices in $X$
and one in $Y$, with Alice to move.  If Alice seeds $X$, Bob forces the last
vertex of $X$ and Alice forces the last vertex of $Y$, giving $(1,1)$.  If
Alice seeds the last vertex of $Y$, the remaining play in $X$ gives $(1,0)$.
Alice is indifferent between these two continuations because she gets one
force in each, but at Bob's preceding node only $(1,1)$ maximises Bob's
coordinate.  Bob therefore retains the reply in $Y$, rather than his
$(2,0)$ reply in $X$.  Hence the branch in which Alice first seeded $Y$ also
contributes $(1,1)$ at $Q_5$, and $\FF(Q_5)=\{(1,1)\}$.

Assume now that $n\geq6$ and that~\eqref{eq:Qn} holds for $n-2$.  If Alice
seeds $X$, Bob can force the other vertex of $X$.  This gives $(0,2)$ when $n$
is even and $(1,1)$ when $n$ is odd, according to which player makes the last
move in $Y$.  These continuations are optimal for Bob: for even $n$ he has the
global maximum of two forces, while for odd $n$ delaying the force allows Alice
to take $X$ and leaves Bob with no force.

If Alice seeds $Y$, Bob has two part choices.  If Bob seeds $X$, Alice forces
the other vertex of $X$.  The resulting pair is $(1,1)$ when $n$ is even and
$(2,0)$ when $n$ is odd.  If Bob seeds $Y$, the resulting position is
strategically equivalent to $Q_{n-2}$.  It is strategically equivalent because
the numbers of already owned twins do not matter, only whether each player owns
a vertex in a part.  By induction this second choice gives $(1,1)$.

For even $n$, Bob gets one force after either of his choices following Alice's
$Y$ seed, so the set-valued rule retains the $(1,1)$ outcome.  Alice prefers
this to the $(0,2)$ outcome produced by seeding $X$.  For odd $n$, Bob strictly
prefers the recursive $(1,1)$ continuation to the $(2,0)$ continuation after
Alice's $Y$ seed; Alice's $X$ seed also gives $(1,1)$.  In either parity the
retained set at $Q_n$ is therefore exactly $\{(1,1)\}$.  This proves
\eqref{eq:Qn} and completes the table.
\end{proof}

\begin{prop}\label{prop:two}
For every $n\geq2$, optimal play on $K_{2,n}$ retains the force pair $(1,1)$.
\end{prop}

\begin{proof}
First take $n\geq4$.  If $n$ is even and Alice begins in $X$, Bob compares the
$XX$ and $XY$ rows of Proposition~\ref{prop:two-table}.  He chooses $XY$
because it gives him two forces rather than one, producing $(0,2)$.  If Alice
begins in $Y$, both of Bob's part choices retain an outcome with pair $(1,1)$.
Alice therefore begins in $Y$, since that gives her one force instead of zero.

If $n$ is odd and Alice begins in $X$, Bob chooses $Y$, since $XY$ gives him
one force and $XX$ gives him none.  If Alice begins in $Y$, Bob again chooses
$Y$, since $YY$ gives him one force and $YX$ gives him none.  Both choices of
Alice therefore lead to $(1,1)$.

For $K_{2,2}$, suppose Alice begins in one part.  If Bob seeds the same part,
Bob later obtains one force and Alice obtains none.  If Bob seeds the opposite
part, each player obtains one force.  Bob has one force in both continuations,
so both survive his tie.  Alice's initial node retains the continuation in
which she also obtains a force.  Thus the retained pair is $(1,1)$.

For $K_{2,3}$, if Alice begins in $X$, Bob obtains a force only by beginning in
$Y$, and the resulting pair is $(1,1)$.  If Alice begins in $Y$, a seed by Bob
in $X$ gives $(2,0)$.  If Bob also seeds $Y$, then two vertices of $X$ and one
vertex of $Y$ are white, with Alice to move.  Seeding $X$ leads to $(1,1)$,
whereas seeding the last vertex of $Y$ leads to $(1,0)$.  Alice gets one force
in both, so this fixed $YY$ position retains both pairs.  At Bob's preceding
node, however, only $(1,1)$ maximises Bob's coordinate.  Bob therefore chooses
the $YY$ branch and retains $(1,1)$.  Hence the initial retained pair is again
$(1,1)$.
\end{proof}

\subsubsection{Conclusion and force allocation}

\begin{thm}\label{thm:complete-bipartite}
For all integers $m,n\geq2$,
\[
Z_g(K_{m,n})=m+n-2.
\]
\end{thm}

\begin{proof}
Assume $2\leq m\leq n$ and let $N=m+n$.  If $m=2$,
Proposition~\ref{prop:two} gives two force moves under optimal play.  If
$m\geq3$ and $N$ is odd, Proposition~\ref{prop:odd-large} gives two forces.  If
$N$ is even and $m=3$, Proposition~\ref{prop:three} gives one force to each
player.  If $N$ is even and $m\geq4$, Proposition~\ref{prop:even-large} gives
two forces to Bob.  These cases exhaust all possibilities, so optimal play
always has
\[
F_A+F_B=2.
\]
Proposition~\ref{prop:force-count} now gives
\[
Z_g(K_{m,n})=N-(F_A+F_B)=m+n-2.
\]
This also agrees with the classical lower bound
$Z_g(K_{m,n})\geq Z(K_{m,n})=m+n-2$.
\end{proof}

The proof yields the stronger allocation statement below.

\begin{cor}
Let $2\leq m\leq n$ and $N=m+n$.  The retained force pair at the initial
position is
\[
(F_A,F_B)=
\begin{cases}
(1,1), & m=2,\\
(2,0), & m\geq3\text{ and }N\text{ is odd},\\
(1,1), & m=3\text{ and }N\text{ is even},\\
(0,2), & m\geq4\text{ and }N\text{ is even}.
\end{cases}
\]
\end{cor}

\begin{remark}[The tie convention is essential]\label{rem:ties}
The set-valued retention of all outcomes tied in the mover's own coordinate is
not merely a bookkeeping choice.  The theorem can fail under a natural-looking
secondary preference.  For example, suppose that whenever a player is tied in
their own force coordinate, that player additionally minimises the opponent's
force coordinate.  In the $Q_4$ analysis above, Bob would then discard
$(1,1)$ in favour of $(0,1)$ after Alice seeds $Y$, and Alice would subsequently
discard $(0,2)$ in favour of $(0,1)$.  Thus $Q_4$ would retain only $(0,1)$.
Propagating this secondary rule through the root analysis of $K_{2,4}$ leaves
the force pair $(0,1)$ uniquely retained.  There would then be only one force
in total, and the resulting game parameter would be $6-1=5$, rather than
$6-2=4=Z(K_{2,4})$.  The conclusion of Theorem~\ref{thm:complete-bipartite} therefore
depends genuinely on the set-valued solution convention defined above.
\end{remark}

\section{Structural results}
We say that \(G\) admits an \emph{alternating two-chain forcing schedule} if
its vertices have an ordered partition
\[
V(G)=\{a_1,\ldots,a_p\}\mathbin{\dot\cup}\{b_1,\ldots,b_q\},
\qquad
p=\left\lceil\frac{|V(G)|}{2}\right\rceil,\quad
q=\left\lfloor\frac{|V(G)|}{2}\right\rfloor,
\]
with the following property. After Alice seeds \(a_1\) and Bob seeds \(b_1\),
the remaining vertices can be coloured in the turn order
\[
a_2,b_2,a_3,b_3,\ldots,
\]
truncated after all vertices are coloured, where \(a_i\) is the unique white
neighbour of \(a_{i-1}\) immediately before Alice's corresponding turn and
\(b_i\) is the unique white neighbour of \(b_{i-1}\) immediately before
Bob's corresponding turn.

\begin{thm}\label{thm:zg2-alternating}
Let \(G\) be a connected graph with order $n\ge 2$. Then \(Z_g(G)=2\) if
and only if \(G\) admits an alternating two-chain forcing schedule.
\end{thm}

\begin{proof}
Suppose first that \(G\) has an alternating two-chain forcing schedule. Alice
and Bob seed \(a_1\) and \(b_1\), respectively, and then perform the indicated
forces. Bob's seed at \(b_1\) is legal: if \(n\geq3\) and Alice has a force
after her first move, its target cannot be \(b_1\), since the schedule requires
\(a_2\) to remain a white neighbour of \(a_1\); for \(n=2\), \(b_1\) is the only
white vertex and the safe-seed exception applies. The Force Preservation Rule
permits every force in the schedule. The outcome is \((1,1)\), so
Lemma~\ref{lem:one-one-retention} gives
\(Z_g(G)=2\).

Conversely, suppose \(Z_g(G)=2\). Some pair in \(\mathcal O(G)\) has total
\(2\), and Lemma~\ref{lem:2} implies that this pair is \((1,1)\). Consider a
play realising that outcome. Every move after the two seeds is a force. Once a
vertex forces its unique white neighbour, it has no white neighbours and can
never force again. The forcing edges of each colour therefore form a single
chronological directed path from that player's seed. Because turns alternate,
the vertices are revealed in the order
\[
a_1,b_1,a_2,b_2,\ldots.
\]
At each force, the next vertex on the corresponding path is the unique white
neighbour of its predecessor. The two paths consequently form an alternating
two-chain forcing schedule.
\end{proof}

\begin{remark}
The ordinary characterization \(Z(G)=2\) by two parallel paths \cite{row2012technique} records the
existence of two forcing chains but not the order in which the two players
must advance them. The alternating schedule in
Theorem~\ref{thm:zg2-alternating} is the extra game-theoretic condition.
\end{remark}

\begin{prop}\label{prop:kn-minor}
If a graph \(G\) has a \(K_n\) minor, then \(Z_g(G)\ge n-1\).
\end{prop}

\begin{proof}
Let \(G\) be a graph with a \(K_n\) minor. Then \(Z(G)\ge n-1\)
\cite{taklimi2013zero}. By Theorem~\ref{prop:ZgZ}, the claim follows immediately.
\end{proof}
Next we show that the game zero forcing number is not minor-monotone.

\begin{prop}\label{prop:deletion-nonmonotone}
Vertex deletion and edge deletion can each either increase or decrease
\(Z_g\), even when the resulting graph is connected. Consequently,
\(Z_g\) is not minor-monotone.
\end{prop}

\begin{proof}
For vertex deletion, let \(R\) be obtained from the claw \(K_{1,3}\), with
centre \(c\) and leaves \(x,y,z\), by subdividing \(cz\) with a new vertex
\(w\). The legal play
\[
\text{Alice seeds }x,\quad \text{Bob seeds }z,\quad
x\to c,\quad z\to w,\quad c\to y
\]
has outcome \((1,1)\). Lemma~\ref{lem:one-one-retention} therefore gives
\(Z_g(R)=2\). Deleting \(z\) leaves a claw, so
\[
Z_g(R-z)=3>2=Z_g(R)
\]
by Proposition~\ref{prop:star}. In the other direction,
Proposition~\ref{prop:complete} gives
\(Z_g(K_4)=3>2=Z_g(K_3)\), and \(K_3\) is obtained by deleting a vertex
of \(K_4\).

For edge deletion, let \(Q\) be the paw graph with triangle \(abc\) and
pendant edge \(cd\). Alice can seed \(a\), Bob can seed \(c\), and then
\(a\to b\) and \(c\to d\) finish with outcome \((1,1)\).
Lemma~\ref{lem:one-one-retention} gives \(Z_g(Q)=2\). Deleting \(ab\) leaves the claw centred at \(c\), so this
edge deletion increases the value from \(2\) to \(3\).

For the reverse direction, let \(D=K_4-e\). Write \(x,y\) for its
nonadjacent vertices of degree two and \(u,v\) for its vertices of degree
three. The legal play
\[
\text{Alice seeds }x,\quad \text{Bob seeds }u,\quad
x\to v,\quad u\to y
\]
has outcome \((1,1)\). Lemma~\ref{lem:one-one-retention} gives
\(Z_g(D)=2\). Thus deleting an edge
from \(K_4\) to obtain \(D\) decreases the value from \(3\) to \(2\).
The first vertex-deletion example also has \(R-z\) as a minor of \(R\)
with larger game value, disproving minor-monotonicity.
\end{proof}

In the next two propositions, we show that subdividing an edge of $G$ can either decrease or increase $Z_g(G)$.

\begin{prop}\label{prop:subdivision-counterexample}
The game zero forcing number is not nondecreasing under edge subdivision.
\end{prop}

\begin{proof}
Let \(G=K_{1,3}\), with centre \(c\) and leaves \(x,y,z\), and subdivide
\(cz\) with a new vertex \(w\) to obtain \(G'\). The graph \(G'\) is the
graph \(R\) from the proof of Proposition~\ref{prop:deletion-nonmonotone},
where Lemma~\ref{lem:one-one-retention} gave \(Z_g(G')=2\). By
Proposition~\ref{prop:star}, \(Z_g(G)=3\). Hence
\[
Z_g(G')=2<3=Z_g(G).
\]
\end{proof}

\begin{prop}\label{prop:subdivision-increase}
Subdividing an edge can increase the game zero forcing number.
\end{prop}

\begin{proof}
Let \(Q\) be the paw graph with triangle \(abc\) and pendant edge \(cd\).
Alice can seed \(a\), Bob can seed \(c\), and the forces
\(a\to b\) and \(c\to d\) give outcome \((1,1)\). Hence
Lemma~\ref{lem:one-one-retention} gives \(Z_g(Q)=2\).

Now subdivide \(cd\) with a new vertex \(w\), obtaining \(Q'\). Alice
seeds \(d\), creating the force \(d\to w\). On Bob's first turn, \(w\)
is the only seed target that would block this force, so the Force Preservation
Rule permits Bob to seed \(a\), \(b\), or \(c\).

Suppose first that Bob seeds \(a\) or \(b\). Alice performs \(d\to w\),
creating the force \(w\to c\). Bob must then seed the other triangle vertex,
which is the only nonblocking seed target, and Alice performs \(w\to c\).
This retained continuation has outcome \((1,2)\): Alice's single seed is
minimum by Lemma~\ref{lem:2}, and Bob necessarily seeds on both of his turns.

If Bob instead seeds \(c\), Alice again performs \(d\to w\). The two
remaining vertices \(a\) and \(b\) are then seeded, one by each player, so
the outcome is \((2,2)\). Alice cannot use only one seed in this branch: a
seed on her third turn would already be her second, while after \(d\to w\)
she has no force to either \(a\) or \(b\) and must later seed one of them.
Thus this continuation is retained at her position. Bob uses two seeds in
every retained branch of his first position, so both types of reply are
optimal for him. At the initial position Alice cannot use fewer than one seed,
and the first branch attains one. Consequently,
\[
(1,2)\in\mathcal O(Q')
\qquad\text{and hence}\qquad
Z_g(Q')\leq3.
\]

It remains to rule out \(Z_g(Q')=2\). By
Theorem~\ref{thm:zg2-alternating}, this would require an alternating
two-chain forcing schedule, with Alice's chain containing three vertices and
Bob's chain containing two. A forcing chain is an induced path, since a chord
would give an earlier forcing vertex at least two white neighbours. The only
partition of \(V(Q')\) into an induced three-vertex path and an edge is, up to
interchanging \(a\) and \(b\),
\[
\{c,w,d\}\mathbin{\dot\cup}\{a,b\}.
\]
Indeed, the other apparent candidate has \(\{a,b,c\}\) inducing a triangle,
not a path.

There are two orientations for Alice's chain. If it begins \(c,w,d\), then
on Alice's first force the vertex \(c\) still has both \(w\) and Bob's
uncoloured triangle vertex as white neighbours. If it begins \(d,w,c\), then
Alice can force \(d\to w\), but Bob's first triangle vertex still has the
other triangle vertex and \(c\) as white neighbours and cannot advance Bob's
chain. Hence no alternating two-chain forcing schedule exists. Therefore
\(Z_g(Q')\neq2\), and so \(Z_g(Q')=3>Z_g(Q)\).
\end{proof}

\section{Concluding remarks}
The preceding sections establish the basic framework and several structural
properties of the game zero forcing number. Many natural questions remain
open, and exact computation is useful both for testing proposed statements
and for identifying plausible directions. This section records the
computational method and evidence used in the paper, then presents the
resulting conjectures and questions.
\subsection{Computational results}
All computational statements in this paper were obtained by exact,
set-valued backward induction. A position is represented by a pair
\((B,R)\) of disjoint vertex sets containing the vertices currently coloured
blue and red. The player to move is determined by the parity of
\(\lvert B\rvert+\lvert R\rvert\).

For a position, let \(W=V(G)\setminus(B\cup R)\), let \(F\subseteq W\) be
the set of targets of the mover's legal forces, and let \(T\subseteq W\) be
the set of targets of the opponent's legal forces. Put
\[
S=W\setminus(T\cup F).
\]
Thus \(S\) is precisely the set of nonblocking seed targets; a target in
\(F\) is an available force move, not an alternative seed move. The legal
target set is
\[
\begin{cases}
W, & T=\varnothing\text{ or }S=\varnothing,\\
S\cup F, & \text{otherwise}.
\end{cases}
\]
Targets in \(F\) are evaluated as force moves and all other permitted targets
as seed moves. This is exactly the clarified Force Preservation Rule,
including permission to execute an own force at a contested target and the
exception when no nonblocking seed is available.

At a terminal position the program returns the residual seed pair
\((0,0)\). Moving backwards through the game tree, a seed move adds one to
the corresponding player's coordinate, while a force move adds nothing. At
an internal position the program takes the union of the resulting pairs from
all legal children and retains every pair that minimises the coordinate
belonging to the player whose turn it is. At the initial position these
residual counts are the final seed totals. Thus the implementation follows
the definition of \(\mathcal O(P)\) without imposing an additional tie-break.
Identical states are memoised, so at most
\(3^{\lvert V(G)\rvert}\) colourings are considered.

The classical value \(Z(G)\) is computed independently by enumerating initial
sets in increasing cardinality and applying the ordinary colour-change rule
to closure. Connected unlabelled graphs were generated once per isomorphism
class with \texttt{geng} from \texttt{nauty/Traces}
\cite{mckay2014practical} and stored in graph6 format. The search uses no
randomized or floating-point computation.

The implementation was calibrated against hand calculations for paths,
cycles, complete graphs, stars, the Petersen graph, and the counterexamples
in this paper. As an additional check, all connected graphs through order
\(5\) were evaluated by a second implementation. The principal exact audits
are summarized below.
\begin{center}
\small
\begin{tabularx}{\textwidth}{@{}>{\raggedright\arraybackslash}Xr>{\raggedright\arraybackslash}X@{}}
\hline
\textbf{Test class} & \textbf{Instances} & \textbf{Result}\\
\hline
Connected graphs of order at most \(9\)
  & \(273{,}192\) & No violation of \(Z_g(G)\leq2Z(G)\)\\
Edge subdivisions from connected graphs of order at most \(8\)
  & \(170{,}884\) & Both increases and decreases occur\\
Edge subdivisions with base minimum degree at least \(2\)
  & \(121{,}685\) & Both increases and decreases occur\\
\(K_{m,n}\), with \(m,n\geq2\) and \(m+n\leq13\)
  & \(30\) & \(Z_g(K_{m,n})=m+n-2\)\\
Connected noncomplete graphs with \(\delta\geq2\), through order \(9\)
  & \(205{,}790\) & No violation of \(Z_g(G)\leq |V(G)|-2\)\\
\hline
\end{tabularx}
\end{center}
\vspace{0.5cm}
In the two subdivision rows, an instance is a pair \((G,e)\), where \(G\)
is one graph6 representative of a connected unlabelled graph and \(e\) ranges
over all edges in that representative. Subdivisions of different edges are
therefore counted separately, even when the resulting graphs are isomorphic.

Among the \(273{,}192\) connected unlabelled graphs of order at most \(9\),
exactly \(20{,}224\) have optimal outcome pairs with more than one total. The
first is the graph in Example~\ref{exam:minimum-necessary}.

The minimum-degree-two subdivision audit contains \(8{,}025\) base graphs.
Of the \(121{,}685\) edge subdivisions, \(7{,}094\) increase \(Z_g\), \(218\)
decrease it, and \(114{,}373\) leave it unchanged. The first decrease has base
graph6 code \texttt{FCurW}; subdividing edge \(04\) produces graph6 code
\texttt{GCUr[O} and changes the value from \(4\) to \(3\). Thus subdivision
monotonicity fails even when the base graph has minimum degree at least two.

The exact solver, audit scripts, and machine-readable result summary are
available in the versioned computational archive at
\url{https://zf.pointdynamics.com/zfgame-computational-archive-v1.0.2.zip}.

\subsection{Conjectures and further questions}

The computations suggest the following conjectures and further questions relating the game zero forcing number to classical zero forcing and graph operations.

\begin{conj}\label{conj:zgz} Every connected graph $G$ with order $n\ge 2$ satisfies
  $$ Z_g(G) \le 2Z(G).$$
\end{conj}

Note that \(Z(P_n)=1\). Thus Proposition~\ref{prop:path} shows that,
if Conjecture~\ref{conj:zgz} is true, then the bound is best possible.
Exact backward induction verifies Conjecture~\ref{conj:zgz} for all
\(273{,}192\) connected unlabelled graphs through order \(9\). The bound is
tight in this range. For example, the order-eight graph with graph6 code
\texttt{G?BDB\{} has \(Z(G)=2\) and \(Z_g(G)=4\).

\begin{example}
Let \(G=H=K_2\). Since \(K_2\vee K_2=K_4\),
\[
Z_g(K_2\vee K_2)=Z_g(K_4)=3
<4=Z_g(K_2)+Z_g(K_2).
\]
\end{example}
However, a computer search shows that for all
\(3{,}037\) unordered pairs of connected unlabelled graphs satisfying
\(\lvert V(G)\rvert,\lvert V(H)\rvert\geq3\) and
\(\lvert V(G)\rvert+\lvert V(H)\rvert\leq10\), it holds that 
\[Z_g(G\vee H)\geq Z_g(G)+Z_g(H).
\] In particular,
\(Z_g(P_4\vee P_4)=5\geq4=Z_g(P_4)+Z_g(P_4)\).
This leads to the following natural question.

\begin{question}\label{ques:join-lower-bound}
If \(G\) and \(H\) are connected graphs with
\(\lvert V(G)\rvert,\lvert V(H)\rvert\geq3\), is
\[
Z_g(G\vee H)\geq Z_g(G)+Z_g(H)?
\]
\end{question}
A possible strengthening is suggested by the same computation.
\begin{question}\label{ques:minimum-degree-order-bound}
If \(G\) is a connected noncomplete graph of order \(n\) with
\(\delta(G)\ge2\), is
\[
Z_g(G)\leq n-2?
\]
\end{question}
The bound holds for all \(205{,}790\) such graphs through order \(9\).
Theorem~\ref{thm:complete-bipartite} shows that the proposed bound would be
sharp on every \(K_{m,n}\) with \(m,n\geq2\).

Finally, we consider the behaviour of the game zero forcing number under clique sums. Clique sums are not additive, even for 1-sums. Identifying one endpoint in
two copies of \(K_2\) gives \(P_3\), and hence
\[
Z_g(P_3)=2<4=Z_g(K_2)+Z_g(K_2).
\]
There is also no formula depending only on the two summand values. Indeed,
identifying one vertex in two copies of \(K_3\) gives the bowtie graph
\(B\). A direct backward-induction check of its two first-move orbits gives
\(\mathcal O(B)=\{(2,1)\}\), and therefore \(Z_g(B)=3\). Since
\(Z_g(K_2)=Z_g(K_3)=2\), the pairs \((K_2,K_2)\) and \((K_3,K_3)\) have the
same summand values, but their 1-sums have game values \(2\) and \(3\),
respectively. In particular,
\[
Z_g(B)=3<4=Z_g(K_3)+Z_g(K_3).
\]
Finding useful bounds for general clique sums remains open.

\bibliographystyle{plain}
\bibliography{references}

\end{document}